\documentclass[final]{amsart}
\usepackage{amssymb,amsmath, color}

\usepackage{amssymb}

\usepackage[notref,notcite]{showkeys}

\newtheorem{Theorem}{Theorem}[section]
\newtheorem{lemma}[Theorem]{Lemma}

\newtheorem{proposition}[Theorem]{Proposition}
\theoremstyle{definition}
\newtheorem{remark}[Theorem]{Remark}
\newtheorem{example}[Theorem]{Example}

\newcommand{\R}{{\mathbb R}}

\DeclareMathOperator{\argmin}{argmin}
\newcommand{\interior}{{\rm int}\kern 0.06em}

\newcommand{\Hb}{\mathcal H}

\newcommand{\Elambdaxi}{\mathcal E_{\lambda,\xi}}
\newcommand{\Elambdap}{\mathcal E_\lambda^p}
\newcommand{\dta}{\frac{2}{3}\alpha}
\newcommand{\dtat}{\frac{2}{3}(\alpha-3)}

\usepackage{ulem}

\begin{document}

\title{ON THE FAST CONVERGENCE OF AN INERTIAL GRADIENT-LIKE DYNAMICS WITH VANISHING VISCOSITY}

\author{Hedy Attouch}

\address{Institut de Math\'ematiques et Mod\'elisation de Montpellier, UMR 5149 CNRS, Universit\'e Montpellier 2, place Eug\`ene Bataillon,
34095 Montpellier cedex 5, France}
\email{hedy.attouch@univ-montp2.fr}

\author{Juan Peypouquet}
\address{Universidad T\'ecnica Federico Santa Mar\'\i a, Av Espa\~na 1680, Valpara\'\i so, Chile}
\email{juan.peypouquet@usm.cl}

\author{Patrick Redont}
\address{Institut de Math\'ematiques et Mod\'elisation de Montpellier, UMR 5149 CNRS, Universit\'e Montpellier 2, place Eug\`ene Bataillon,
34095 Montpellier cedex 5, France}
\email{patrick.redont@univ-montp2.fr}

\date{June 11, 2015}



\begin{abstract}
In a real Hilbert space  $\mathcal H$,  we study the fast convergence properties as $t \to + \infty$ of the trajectories of the second-order evolution equation
$$
\ddot{x}(t) + \frac{\alpha}{t} \dot{x}(t) + \nabla \Phi (x(t)) = 0,
$$
where $\nabla \Phi$ is the gradient of a convex continuously differentiable function $\Phi : \mathcal H \rightarrow \mathbb R$, and $\alpha$ is a positive parameter. In this inertial system, the viscous  damping coefficient $\frac{\alpha}{t}$ vanishes asymptotically in a moderate way. For $\alpha > 3$, we show that any trajectory converges weakly to a minimizer of $\Phi$, just assuming that $\argmin \Phi \neq \emptyset$. The strong convergence is established in various  practical situations. These results complement the $\mathcal O(t^{-2})$ rate of convergence for the values obtained by Su, Boyd and Cand\`es. Time discretization of this system, and some of its variants, provides new fast converging algorithms, expanding the field of rapid methods for structured convex minimization introduced by Nesterov, and further developed by Beck and Teboulle. This study also complements recent advances due to Chambolle and Dossal.
\end{abstract}

\thanks{Effort sponsored by the Air Force Office of Scientific Research, Air Force Material Command, USAF, under grant number FA9550-14-1-0056. Also supported by Fondecyt Grant 1140829, Conicyt Anillo ACT-1106, ECOS-Conicyt Project C13E03, Millenium Nucleus ICM/FIC RC130003, Conicyt Project MATHAMSUD 15MATH-02, Conicyt Redes 140183, and Basal Project CMM Universidad de Chile.}

\maketitle

\vspace{0.3cm}

\section{Introduction}

Let $\mathcal H$ be a real Hilbert space, which is endowed with the scalar product $\langle \cdot,\cdot\rangle$ and norm $\|\cdot\|$, and let $\Phi : \mathcal H \rightarrow \mathbb R$ be a (smooth) convex function. In this paper, we study the solution trajectories of the second-order differential equation 
\begin{equation}\label{edo01}
\ddot{x}(t) + \frac{\alpha}{t} \dot{x}(t) + \nabla \Phi (x(t)) = 0,
\end{equation}
with $\alpha>0$, in terms of their asymptotic behavior as $t \to + \infty$. Although this is not our main concern, we point out that, given $t_0>0$, for any $x_0 \in \mathcal H,  \ v_0 \in \mathcal H $, the existence of a unique global solution on $[t_0,+\infty[$ for the Cauchy problem with initial condition $x(t_0)=x_0$ and $\dot x(t_0)=v_0$ can be guaranteed, for instance, if $\nabla\Phi$ is Lipschitz-continuous on bounded sets.

\medskip

The importance of this evolution system is threefold:

\medskip

{\it 1. Mechanical interpretation:} It describes the position of a particle subject to a potential energy function $\Phi$, and an isotropic linear damping with a viscosity parameter that vanishes asymptotically. This provides a simple model for a progressive reduction of the friction, possibly due to material fatigue.

\bigskip

{\it 2. Fast minimization of function $\Phi$:} Equation \eqref{edo01} is a particular case of
the inertial gradient-like system 
\begin{equation}\label{edo2}
\ddot{x}(t) + a(t) \dot{x}(t) + \nabla \Phi (x(t)) = 0,
\end{equation}
with {\it asymptotic vanishing damping}, studied by Cabot, Engler and Gadat in \cite{CEG1,CEG2}. As shown in \cite[Corollary 3.1]{CEG1} (under some additional conditions on $\Phi$), every solution $x(\cdot)$ of \eqref{edo2} satisfies $\lim_{t\to+\infty}\Phi(x(t))=\min\Phi$, provided $\int_0^{\infty}a(t)dt = + \infty$. The specific case \eqref{edo01} was studied by Su, Boyd and Cand\`es in \cite{SBC} in terms of the rate of convergence for the values. More precisely, \cite[Theorem 4.1]{SBC} establishes that 
$\Phi(x(t))-\min\Phi =\mathcal O(t^{-2})$, whenever $\alpha\ge 3$. Unfortunately, their analysis does not entail the convergence of the trajectory itself.

\bigskip

{\it 3. Relationship with fast numerical optimization methods:} As pointed out in \cite[Section 2]{SBC}, for $\alpha=3$, \eqref{edo01} can be seen as a continuous version of the fast convergent method of Nesterov (see \cite{Nest1,Nest2,Nest3,Nest4}), and its widely used successors, such as the {\it Fast Iterative Shrinkage-Thresholding Algorithm} (FISTA), studied in \cite{BT}. These methods have a convergence rate of $\Phi(x^k)-\min\Phi=\mathcal O(k^{-2})$, where $k$ is the number of iterations. As for the continuous-time system \eqref{edo01}, convergence of the sequences generated by FISTA and related methods has not been established so far. This is a central and long-standing question in the study of numerical optimization methods.

\bigskip

The purpose of this research is to establish the convergence of the trajectories satisfying \eqref{edo01}, as well as the sequences generated by the corresponding numerical methods with Nesterov-type acceleration. We also complete the study with several stability properties concerning both the continuous-time system and the algorithms.

\bigskip

More precisely, the main contributions of this work are the following: In Section \ref{S:minimizing}, we first establish the minimizing property in the general case where $\alpha>0$ and $\inf\Phi$ is not necessarily attained. As a consequence, every weak limit point of the trajectory must be a minimizer of $\Phi$, and so, the existence of a bounded trajectory characterizes the existence of minimizers. Next, assuming $\argmin\Phi\neq\emptyset$ and $\alpha\ge 3$, we recover the $\mathcal O(t^{-2})$ convergence rates and give several examples and counterexamples concerning the optimality of these results. Next, we show that every solution of \eqref{edo01} converges weakly to a minimizer of $\Phi$ provided $\alpha >3$ and $\argmin \Phi \neq \emptyset$. We rely on a Lyapunov analysis, which was first used by Alvarez \cite{Al} in the context of the heavy ball with friction. For the limiting case $\alpha=3$, which corresponds exactly to Nesterov's method, the convergence of the tra!
 jectories is still a puzzling open question. We finish this section by providing an ergodic convergence result for the acceleration of the system in case $\nabla\Phi$ is Lipschitz-continuous on sublevel sets of $\Phi$. Strong convergence is established in various practical situations enjoying further geometric features, such as strong convexity, symmetry, or nonemptyness of the interior of the solution set (see Section \ref{S:strong}). In the strongly convex case, we obtained a surprising result: convergence of the values occurs at a rate of $\mathcal O(t^{-\frac{2\alpha}{3}})$. Section \ref{S:algorithm} contains the analogous results for the associated Nesterov-type algorithms (which also correspond to the case $\alpha>3$). As we were preparing the final version of this manuscript, we discovered the preprint \cite{CD} by Chambolle and Dossal, where the weak convergence result is obtained by a similar, but different, argument (see \cite[Theorem 3]{CD}).


\section{Minimizing property, convergence rates and weak convergence of the trajectories}\label{S:minimizing}

We begin this section by providing some preliminary estimations concerning the global energy of the system \eqref{edo01} and the distance to the minimizers of $\Phi$. These allow us to show the minimizing property of the trajectories under minimal assumptions. Next, we recover the convergence rates for the values originally given in \cite{SBC} and obtain further decay estimates that ultimately imply the convergence of the solutions of \eqref{edo01}. We finish the study by proving an ergodic convergence result for the acceleration. Several examples and counterexamples are given throughout the section.

\subsection{Preliminary remarks and estimations}

The existence of global solutions to \eqref{edo01} has been examined, for instance, in \cite[Proposition 2.2.]{CEG1} in the case of a general asymptotic vanishing damping coefficient. In our setting, for any $t_0 >0$, $\alpha >0$, and  $(x_0,v_0)\in \mathcal H\times\mathcal H$, there exists a unique global solution $x : [t_0, +\infty[ \rightarrow \mathcal H$ of \eqref{edo01}, satisfying the initial condition $x(t_0)= x_0$, $\dot{x}(t_0)= v_0$, under the sole assumption that $\inf\Phi > - \infty$. Taking $t_0>0$ comes from the singularity of the damping coefficient $a(t) = \frac{\alpha}{t}$ at zero. Indeed, since we are only concerned about the asymptotic behavior of the trajectories, we do not really care about the origin of time. If one insists in starting from $t_0 =0$, then all the results remain valid with $a(t)=\frac{\alpha}{t+1}$.

\bigskip

At different points, we shall use the {\it global energy} of the system, given by $W:[t_0,+\infty [\rightarrow\R$
\begin{equation}\label{E:W}
W(t) = \frac{1}{2}\|\dot{x}(t)\|^2 + \Phi (x(t)).
\end{equation}

Using \eqref{edo01}, we immediately obtain
\begin{lemma}\label{L:W}
Let $W$ be defined by \eqref{E:W}. For each $t>t_0$, we have
$$\dot{W}(t) = -\frac{\alpha}{t}\|\dot{x}(t)\|^2.$$ 
Hence, $W$ is nonincreasing\footnote{In fact, $W$ decreases strictly, as long as the trajectory is not stationary.}, and $W_\infty=\lim_{t\rightarrow+\infty}W(t)$ exists in $\R\cup\{-\infty\}$. If $\Phi$ is bounded from below, $W_\infty$ is finite.
\end{lemma}

Now, given $z\in\mathcal H$, we define $h_z:[t_0,+\infty [ \to\R$ by
\begin{equation} \label{E:h_z}
h_z(t)=\frac{1}{2}\|x(t)-z\|^2.
\end{equation} 

By the Chain Rule, we have
$$\dot{h}_z(t) = \langle x(t) - z , \dot{x}(t)  \rangle\qquad\hbox{and}\qquad 
\ddot{h}_z(t) = \langle x(t) - z , \ddot{x}(t)  \rangle + \| \dot{x}(t) \|^2.$$
Using \eqref{edo01}, we obtain
\begin{equation} \label{wconv3}
\ddot{h}_z(t) + \frac{\alpha}{t} \dot{h}_z(t) =  \| \dot{x}(t) \|^2 + \langle x(t) - z , \ddot{x}(t) + \frac{\alpha}{t} \dot{x}(t)  \rangle =  \| \dot{x}(t) \|^2 +  \langle x(t) - z , -\nabla \Phi (x(t)) \rangle .
\end{equation} 
The convexity of $ \Phi$ implies
$$ \langle x(t) - z , \nabla \Phi (x(t))  \rangle \geq  \Phi (x(t)) - \Phi (z),$$
and we deduce that
\begin{equation} \label{E:ineq_h_z}
\ddot{h}_z(t) + \frac{\alpha}{t} \dot{h}_z(t) +\Phi (x(t)) - \Phi (z)  \leq \|\dot{x}(t)\|^2.
\end{equation} 

We have the following relationship between $h_z$ and $W$:

\begin{lemma}\label{L:Wh_z}
Take $z\in \mathcal H$, and let $W$ and $h_z$ be defined by \eqref{E:W} and \eqref{E:h_z}, respectively. There is a constant $C$ such that $$\int_{t_0}^t \frac{1}{s}\left(W(s) - \Phi(z)\right)ds  
		\leq C-\frac{1}{t} \dot{h}_{z}(t)-\frac{3}{2\alpha}W(t).$$
\end{lemma}

\begin{proof}
Divide \eqref{E:ineq_h_z} by $t$, and use the definition of $W$ given in \eqref{E:W}, to obtain
$$\frac{1}{t} \ddot{h}_{z}(t) + \frac{\alpha}{t^2} \dot{h}_{z}(t) + \frac{1}{t}\left(  W(t) - \Phi (z) \right)  
\leq \frac{3}{2t}\|\dot{x}(t)\|^2.$$
Integrate this expression from $t_0$ to $t>t_0$ (use integration by parts for the first term), to obtain 
\begin{equation}\label{conv3-8}
\int_{t_0}^t \frac{1}{s}\left(  W(s) - \Phi (z) \right)ds  \leq 
  \frac{1}{t_0} \dot{h}_{z}(t_0) - \frac{1}{t} \dot{h}_{z}(t) - 
(\alpha  + 1)\int_{t_0}^t \frac{1}{s^2}\dot{h}_{z}(s)ds +
\int_{t_0}^t \frac{3}{2s}\|\dot{x}(s)\|^2 ds.
\end{equation}
On the one hand, Lemma \ref{L:W} gives
$$
\int_{t_0}^t \frac{3}{2s}\|\dot{x}(s)\|^2 ds = 
\frac{3}{2\alpha}(W(t_0)-W(t)).
$$
On the other hand, another integration by parts yields 
$$\int_{t_0}^t \frac{1}{s^2}\dot{h}_{z}(s)ds = 
\frac{1}{t^2}h_z(t) - \frac{1}{t_0^2}h_z(t_0) + 
  \int_{t_0}^t\frac{2}{s^3}h_z(s)ds 
\geq - \frac{1}{t_0^2}h_z(t_0).$$
Combining these inequalities with \eqref{conv3-8}, we get 
$$
\int_{t_0}^t \frac{1}{s}\left(  W(s) - \Phi (z) \right)ds  
\leq 
  \frac{1}{t_0} \dot{h}_{z}(t_0) - \frac{1}{t} \dot{h}_{z}(t) + 
(\alpha  + 1)\frac{1}{t_0^2}h_z(t_0) + 
\frac{3}{2\alpha}(W(t_0)-W(t))
=
C-\frac{1}{t} \dot{h}_{z}(t)-\frac{3}{2\alpha}W(t),$$
where $C$ collects the constant terms.
\end{proof}

\subsection{Minimizing property}

It turns out that the trajectories of \eqref{edo01} minimize $\Phi$ in the completely general setting, where $\alpha>0$, $\argmin \Phi$ is possibly empty and $\Phi$ is not necessarily bounded from below. This property was obtained by Alvarez in \cite[Theorem 2.1]{Al} for the heavy ball with friction (where the damping is constant). Similar results can be found in \cite{CEG1}.

We have the following:

\begin{Theorem} \label{Thm-weak-conv2}
Let $\alpha>0$ and suppose $x:[t_0,+\infty[\rightarrow\Hb$ is a solution of \eqref{edo01}. Then 
\begin{enumerate}
	\item [i)] $W_\infty=\lim_{t \rightarrow+\infty}W(t)=\lim_{t \rightarrow+\infty}\Phi(x(t))=\inf\Phi\in\R\cup\{-\infty\}$.
	\item [ii)] As $t\to+\infty$, every weak limit point of $x(t)$ lies in $\argmin\Phi$. 
	\item [iii)] If $\argmin\Phi=\emptyset$, then $\lim_{t\to+\infty}\|x(t)\|=+\infty$.
	\item [iv)] If $x$ is bounded, then $\argmin\Phi\neq\emptyset$.
	\item [v)] If $\Phi$ is bounded from below, then $\lim_{t\to+\infty}\|\dot x(t)\|=0$. 
	\item [vi)] If $\Phi$ is bounded from below and $x$ is bounded, then $\lim_{t\to+\infty}\dot h_z(t)=0$ for each $z\in\mathcal H$. Moreover,  
  	$$\int_{t_0}^\infty\frac{1}{t}(\Phi(x(t))-\min\Phi)dt<+\infty.$$
\end{enumerate}
\end{Theorem}

\begin{proof}
To prove i), first set $z\in\Hb$ and $\tau\geq t>t_0$. By Lemma \ref{L:W}, $W$ in nonincreasing. Hence, Lemma \ref{L:Wh_z} gives
$$(W(\tau)-\Phi(z))\int_{t_0}^t\frac{ds}{s}+\frac{3}{2\alpha}W(\tau) \leq 
C-\frac{1}{t}\dot h_z(t),$$
which we rewrite as
$$
(W(\tau)-\Phi(z))\left(\int_{t_0}^t\frac{ds}{s}+\frac{3}{2\alpha}\right) \leq 
  C-\frac{3}{2\alpha}\Phi(z)-\frac{1}{t}\dot h_z(t),
$$
and then
$$
(W(\tau)-\Phi(z))\left(\ln(t)+\frac{3}{2\alpha}-\ln(t_0)\right) \leq 
  C-\frac{3}{2\alpha}\Phi(z)-\frac{1}{t}\dot h_z(t).
$$
Integrate from $t=t_0$ to $t=\tau$ to obtain
$$
(W(\tau)-\Phi(z))
  \left(
  \tau\ln(\tau)-t_0\ln(t_0)+t_0 -\tau+
    \left(\frac{2}{3\alpha}-\ln(t_0)\right)(\tau-t_0)\right) 
  \leq 
  \left(C-\frac{3}{2\alpha}\Phi(z)\right)(\tau-t_0)-
  \int_{t_0}^\tau\frac{1}{t}\dot h_z(t)dt.
$$
But
$$
\int_{t_0}^\tau\frac{\dot h_z(t)}{t}dt
  =\frac{h_z(\tau)}{\tau}-\frac{h_z(t_0)}{t_0}+\int_{t_0}^\tau\frac{h_z(t)}{t^2}dt
  \geq-\frac{h_z(t_0)}{t_0}.
$$
Hence, 
$$
(W(\tau)-\Phi(z))(\tau\ln(\tau)+A\tau+B)
  \leq \widetilde C\tau+D,
$$
for suitable constants $A$, $B$, $\widetilde C$ and $D$. This immediately yields $W_\infty\leq\Phi(z)$, and hence $W_\infty\leq \inf\Phi$. It suffices to observe that 
$$\inf\Phi\leq \liminf_{t\to+\infty}\Phi(x(t))\leq  \limsup_{t\to+\infty}\Phi(x(t))\leq \lim_{t\to+\infty}W(t)=W_\infty$$ 
to obtain i). 

Next, ii) follows from i) by the weak lower-semicontinuity of $\Phi$. Clearly, iii) and iv) are immediate consequences of ii). We obtain v) by using i) and the definition of $W$ given in \eqref{E:W}. For vi), since $\dot h_z(t)=\langle x(t) - z , \dot{x}(t)  \rangle$ and $x$ is bounded, v) implies $\lim\limits_{t\to+\infty}\dot h_z(t)=0$. Finally, using the definition of $W$ together with Lemma \ref{L:Wh_z} with $z\in\argmin\Phi$, we get
$$\int_{t_0}^\infty\frac{1}{t}(\Phi(x(t))-\min\Phi)dt\le C-\frac{3}{2\alpha}\min\Phi<+\infty,$$
which completes the proof.
\end{proof}

\begin{remark}
We shall see in Theorem \ref{T:alpha3} that, for $\alpha\ge 3$, the existence of minimizers implies that every solution of \eqref{edo01} is bounded. This gives a converse to part iv) of Theorem \ref{Thm-weak-conv2}.
\end{remark}

If $\Phi$ is not bounded from below, it may be the case that $\|\dot x(t)\|$ does not tend to zero, as shown in the following example:

\begin{example}
Let $\mathcal H=\R$ and $\alpha>0$. The function $x(t)=t^2$ satisfies \eqref{edo01} with $\Phi(x)=-2(\alpha+1)x$. Then $\lim_{t\to+\infty}\Phi(x(t))=-\infty=\inf\Phi$, and $\lim_{t\to+\infty}\|\dot x(t)\|=+\infty$.
\end{example}

\subsection{Two ``anchored" energy functions}\label{DefEnerg}

We begin by introducing two important auxiliary functions and showing their basic properties. From now on, we assume $\argmin\Phi\neq\emptyset$. Fix 
$\lambda\geq0$, $\xi\geq0$, $p\geq0$ and $x^*\in\argmin\Phi$. Let $x:[t_0,+\infty [\to\mathcal H$ be a solution of \eqref{edo01}. For $t\geq t_0$ define
\begin{eqnarray*}
\Elambdaxi(t) & = & 
  t^2(\Phi(x(t))-\min\Phi)+\frac{1}{2}\|\lambda(x(t)-x^*)+t\dot x(t)\|^2+
  \frac{\xi}{2}\|x(t)-x^*\|^2, \\
\Elambdap(t) & = & 
  t^p\mathcal E_{\lambda,0}(t) \; = \;
  t^p\left(t^2(\Phi(x(t))-\min\Phi)+ 
    \frac{1}{2}\|\lambda(x(t)-x^*)+t\dot x(t)\|^2\right),
\end{eqnarray*}
and notice that $\Elambdaxi$ and $\Elambdap$ are sums of nonnegative terms. These generalize the energy functions $\mathcal E$ and $\tilde{\mathcal E}$ introduced in \cite{SBC}. More precisely, $\mathcal E=\mathcal E_{\alpha-1,0}$ and $\tilde{\mathcal E}=\mathcal E_{(2\alpha-3)/3}^1$.

We need some preparatory calculations prior to differentiating $\Elambdaxi$ 
and $\Elambdap$. For simplicity of notation, we do not make the dependence of 
$x$ or $\dot x$ on $t$ explicit. Notice that we use \eqref{edo01} in the 
second line to dispose of $\ddot x$. 
\begin{eqnarray*}
\frac{d}{dt}t^2(\Phi(x)-\min\Phi) & = & 
  2t(\Phi(x)-\min\phi)+t^2\langle\dot x,\nabla\Phi(x)\rangle \\
\frac{d}{dt}\frac{1}{2}\|\lambda(x-x^*)+t\dot x\|^2 & = & 
  - \lambda t\langle x-x^*,\nabla\Phi (x)\rangle 
  - \lambda(\alpha-\lambda-1)\langle x-x^*,\dot x\rangle 
  - (\alpha-\lambda-1)t\|\dot x\|^2 - t^2\langle\dot x,\nabla\Phi(x)\rangle \\
\frac{d}{dt}\frac{1}{2}\|x-x^*\|^2 & = & \langle x-x^*,\dot x\rangle.
\end{eqnarray*}
Whence, we deduce 
\begin{eqnarray}
\label{dElambdaxi}
\frac{d}{dt}\Elambdaxi(t) & = & 
  2t(\Phi(x)-\min\Phi) - \lambda t\langle x-x^*,\nabla\Phi(x)\rangle 
  + (\xi-\lambda(\alpha-\lambda-1))\langle x-x^*,\dot x\rangle 
  - (\alpha-\lambda-1)t\|\dot x\|^2 \\
\label{dElambdap}
\frac{d}{dt}\Elambdap(t) & = & 
  (p+2)t^{p+1}(\Phi(x)-\min\Phi) 
  - \lambda t^{p+1}\langle x-x^*,\nabla\Phi(x)\rangle
  - \lambda(\alpha-\lambda-1-p)t^p\langle x-x^*,\dot x\rangle \\
\nonumber
 & & 
  + \frac{\lambda^2 p}{2}t^{p-1}\|x-x^*\|^2
  - \left(\alpha-\lambda-1-\frac{p}{2}\right)t^{p+1}\|\dot x\|^2.
\end{eqnarray}

\begin{remark}\label{R:E_decreasing}
If $y\in\mathcal H$ and $x^*\in\argmin\Phi$, the convexity of $\Phi$ gives $\min\Phi=\Phi(x^*)\ge\Phi(y)+\langle\nabla\Phi(y),x^*-y\rangle$. Using this in \eqref{dElambdaxi} with $y=x(t)$, we obtain 
$$\frac{d}{dt}\Elambdaxi(t) \leq 
  (2-\lambda)\,t\,(\Phi(x)-\min\Phi) + (\xi-\lambda(\alpha-\lambda-1))\langle x-x^*,\dot x\rangle 
  - (\alpha-\lambda-1)\,t\,\|\dot x\|^2.$$
If one chooses $\xi^*=\lambda(\alpha-\lambda-1)$, then
$$\frac{d}{dt}\mathcal E_{\lambda,\xi^*}(t) \leq(2-\lambda)\,t\,(\Phi(x)-\min\Phi) - (\alpha-\lambda-1)\,t\,\|\dot x\|^2.$$
Therefore, if $\alpha\ge 3$ and $2\le\lambda\le\alpha-1$, then $\mathcal E_{\lambda,\xi^*}$ is nonincreasing. The extreme cases $\lambda=2$ and $\lambda=\alpha-1$ are of special importance, as we shall see shortly.
\end{remark}

\subsection{Rate of convergence for the values}

We now recover convergence rate results for the value of $\Phi$ along a trajectory, already established in \cite[Theorem 4.1]{SBC}:

\begin{Theorem}\label{T:SBC}
Let $x:[t_0,+\infty[ \to\mathcal H$ be a solution of \eqref{edo01} and assume $\argmin\Phi\neq\emptyset$. If $\alpha\ge 3$, then 
$$\Phi(x(t))-\min\Phi \leq \frac{\mathcal E_{\alpha-1,0}(t_0)}{t^2}.$$
If $\alpha>3$, then
$$\int_{t_0}^{+\infty}t\big(\Phi(x(t))-\min\Phi\big)\,dt\le \frac{\mathcal E_{\alpha-1,0}(t_0)}{\alpha-3}<+\infty.$$
\end{Theorem}

\begin{proof}
Suppose $\alpha\ge 3$. Choose $\lambda=\alpha-1$ and $\xi=0$, so that $\xi-\lambda(\alpha-\lambda-1)=\alpha-\lambda-1=0$ and $\lambda-2=\alpha-3$. Remark \ref{R:E_decreasing} gives
\begin{equation} \label{DE1}
\frac{d}{dt}\mathcal E_{\alpha-1,0}(t) \leq  - (\alpha-3)\,t\,(\Phi(x)-\min\Phi),
\end{equation} 
and $\mathcal E_{\alpha-1,0}$ is nonincreasing. Since
$t^2(\Phi(x)-\min\Phi)\leq \mathcal E_{\alpha-1,0}(t)$, we obtain
$$\Phi(x(t))-\min\Phi \leq \frac{\mathcal E_{\alpha-1,0}(t_0)}{t^2}.$$ 
If $\alpha>3$, integrating \eqref{DE1} from $t_0$ to $t$ we obtain 
$$\int_{t_0}^t s(\Phi(x(s))-\min\Phi)ds \leq 
  \frac{1}{\alpha-3}(\mathcal E_{\alpha-1,0}(t_0)-\mathcal E_{\alpha-1,0}(t))\le \frac{1}{\alpha-3}\mathcal E_{\alpha-1,0}(t_0),$$
which allows us to conclude.
\end{proof}

\begin{remark} 
It would be interesting to know whether $\alpha=3$ is critical for the convergence rate given above.
\end{remark}

\begin{remark}
For the (first-order) steepest descent dynamical system, the typical rate of convergence is $\mathcal O(1/t)$ (see, for instance, \cite[Section 3.1]{Pey_Sor}). For the second-order system \eqref{edo01}, we have obtained a rate of $\mathcal O(1/t^2)$. It would be interesting to know whether higher-order systems give the corresponding rates of convergence. Another challenging question is the convergence rate of the trajectories defined by differential equations involving fractional time derivatives, as well as integro-differential equations.
\end{remark}

\subsection{Some examples and counterexamples}

A convergence rate of $\mathcal O(1/t^2)$ may be attained, even if $\argmin \Phi = \emptyset$ and $\alpha<3$. This is illustrated in the following example:

\begin{example}
Let $\mathcal H = \mathbb R$ and take $\Phi (x) = \frac{\alpha -1}{2}e^{-2x}$ 
with $\alpha\geq1$. Let us verify that $x(t)= \ln t$ is a solution of 
(\ref{edo01}). On the one hand,
$$
\ddot{x}(t) + \frac{\alpha}{t} \dot{x}(t) = \frac{\alpha -1}{t^2} .
$$
On the other hand, $\nabla \Phi (x)= -(\alpha -1)e^{-2x}$ which gives
$$
 \nabla \Phi (x(t))= -(\alpha -1)e^{-2 \ln t}=  -\frac{\alpha -1}{t^2} .
$$
Thus, $x(t)= \ln t$ is a solution of (\ref{edo01}). Let us examine the minimizing property. We have $\inf \Phi = 0$, and
$$
\Phi (x(t)) = \frac{\alpha -1}{2}e^{-2 \ln t}= \frac{\alpha -1}{t^2} .
$$
\end{example}

Therefore, one may wonder whether the rapid convergence of the values is true in general. The following example shows that this is not the case:

\begin{example}\label{X:xtheta}
Let $\mathcal H = \mathbb R$ and take $\Phi (x) = \frac{c}{x^{\theta}}$ , with 
$\theta>0$ , $\alpha\geq\frac{\theta}{(2+\theta)}$ and 
$c= \frac{2( 2\alpha + \theta (\alpha -1))}{\theta ( 2+ \theta )^2}$.
Let us verify that $x(t)=  t^{\frac{2}{2 + \theta}}$ is a solution of 
\eqref{edo01}. On the one hand,
$$
\ddot{x}(t) + \frac{\alpha}{t} \dot{x}(t) = \frac{2}{(2+ \theta)^2}(2\alpha + \theta (\alpha -1)) t^{- \frac{2(1+ \theta)}{2 + \theta}}.
$$
On the other hand, $\nabla \Phi (x)= -c \theta x^{-\theta -1} $ which gives
$$
 \nabla \Phi (x(t))= -c \theta t^{- \frac{2(1+ \theta)}{2 + \theta}}= -\frac{2}{(2+ \theta)^2}(2\alpha + \theta (\alpha -1)) t^{- \frac{2(1+ \theta)}{2 + \theta}}.
$$
Thus, $x(t)=  t^{\frac{2}{2 + \theta}}$ is solution of \eqref{edo01}. Let us examine the minimizing property. We have $\inf \Phi = 0$, and
$$
\Phi (x(t)) = c \frac{1}{t^{\frac{2 \theta}{2 + \theta}}}\ , 
  \mbox{ with }\ \frac{2 \theta}{2 + \theta}<2.
$$
\end{example}

We conclude that the order of convergence may be strictly slower than $\mathcal O(1/t^2)$ when $\argmin\Phi=\emptyset$. In the Example \ref{X:xtheta}, this occurs no matter how large $\alpha$ is. The speed of convergence of $\Phi (x(t))$ to $\inf \Phi$ depends on the behavior of $\Phi (x)$ as $\| x\| \to +\infty$. The above examples suggest that, when $\Phi (x)$ decreases rapidly and attains its infimal value as $\| x\| \to \infty$, we can expect fast convergence of $\Phi (x(t))$.

\bigskip

Even when $\argmin\Phi\neq\emptyset$, $\mathcal O (1/t^2)$ is the worst possible case for the rate of convergence, attained as a limit in the following example:

\begin{example}
Take $\mathcal H = \mathbb R$ and $\Phi (x) = c|x|^{\gamma}$, where $c$ and $\gamma$ are positive parameters. Let us look for nonnegative solutions of \eqref{edo01} of the form $x(t)= \frac{1}{t^{\theta}}$, with $\theta >0$. This means that the trajectory is not oscillating, it is a completely damped trajectory. We begin by determining the values of $c$, $\gamma$ and $\theta$ that provide such solutions. On the one hand,
$$
\ddot{x}(t) + \frac{\alpha}{t} \dot{x}(t) = \theta (\theta +1 -\alpha) \frac{1}{t^{\theta+ 2}}.
$$
On the other hand, $\nabla \Phi (x)= c \gamma |x|^{\gamma -2}x $, which gives
$$
\nabla \Phi (x(t))=  c \gamma \frac{1}{t^{\theta (\gamma -1)}}.
$$
Thus, $x(t)= \frac{1}{t^{\theta}}$ is solution of \eqref{edo01} if, and only if,
\begin{itemize}
	\item [i)] $\theta+ 2 = \theta (\gamma -1)$, which is equivalent to $\gamma >2$ and $\theta= \frac{2}{\gamma -2}$; and
	\item [ii)] $c \gamma = \theta (\alpha -\theta -1)$, which is equivalent to $\alpha > \frac{\gamma}{\gamma -2}$ and $c= \frac{2}{\gamma(\gamma -2)}( \alpha - \frac{\gamma}{\gamma -2})$.
\end{itemize}
We have $\min\Phi = 0$ and
$$
\Phi (x(t)) =\frac{2}{\gamma(\gamma -2)}( \alpha - \frac{\gamma}{\gamma -2}  )  \frac{1}{t^{\frac{2 \gamma}{\gamma -2 }}}.
$$

The speed of convergence of $\Phi (x(t))$ to $0$ depends on the parameter $\gamma$. As $\gamma$ tends to infinity, the exponent $\frac{2 \gamma}{\gamma -2 }$ tends to $2$. This limiting situation is obtained by taking  a function $\Phi$ that becomes very flat around the set of its minimizers. Therefore, without other geometric assumptions on $\Phi$, we cannot expect a convergence rate better than $\mathcal O (1/t^2)$.
By contrast, in Section \ref{S:strong}, we will show better rates of convergence under some geometrical assumptions, like strong convexity of $\Phi$.
\end{example}

\subsection{Weak convergence of the trajectories}

In this subsection, we show the convergence of the solutions of \eqref{edo01}, provided $\alpha>3$. We begin by establishing some preliminary estimations that cannot be derived from the analysis carried out in \cite{SBC}. The first statement improves part v) of Theorem \ref{Thm-weak-conv2}, while the second one is the key to proving the convergence of the trajectories of \eqref{edo01}:

\begin{Theorem}\label{T:alpha3}
Let $x:[t_0,+\infty[ \to\mathcal H$ be a solution of \eqref{edo01} with $\argmin\Phi\neq\emptyset$. 
\begin{itemize}
	\item [i)] If $\alpha\ge 3$ and $x$ is bounded, then
	$\|\dot x(t)\|=\mathcal O(1/t)$. More precisely, 
	\begin{equation}\label{bornedx}
	\|\dot x(t)\| \leq \frac{1}{t}  \left(\sqrt{2\mathcal E_{\alpha-1,0}(t_0)} + (\alpha-1)\sup_{t\geq t_0}\|x(t)-x^*\|\right).
	\end{equation}
	\item [ii)] If $\alpha>3$, then $x$ is bounded and 
	\begin{equation}\label{energy1}
	\int_{t_0}^{+\infty}t\|\dot x(t)\|^2\,dt\le \frac{\mathcal E_{2,2(\alpha-3)}(t_0)}{\alpha-3}<+\infty.
	\end{equation}
\end{itemize}
\end{Theorem}

\begin{proof}
To prove i), assume $\alpha\ge 3$ and $x$ is bounded. From the definition of $\Elambdaxi$, we have $\frac{1}{2}\|\lambda(x-x^*)+t\dot x\|^2\leq\Elambdaxi(t)$, and so $\|t\dot x\| \leq \sqrt{2\Elambdaxi(t)}+\lambda\|x-x^*\|$. By Remark \ref{R:E_decreasing}, $\mathcal E_{\alpha-1,0}$ is nonincreasing, and we immediately obtain \eqref{bornedx}.

In order to show ii), suppose now that $\alpha>3$. Choose $\lambda=2$ and $\xi^*=2(\alpha-3)$. By Remark \ref{R:E_decreasing}, we have
\begin{equation}\label{DE2}
\frac{d}{dt}\mathcal E_{\lambda,\xi^*}(t) \leq - (\alpha-3)\,t\,\|\dot x\|^2,
\end{equation}
and $\mathcal E_{\lambda,\xi^*}$ is nonincreasing. From the definition of $\Elambdaxi$, we deduce that 
$\|x(t)-x^*\|^2 \leq \frac{2}{\xi} \Elambdaxi (t)$, which gives
\begin{equation} \label{E:x_bounded}
\|x(t)-x^*\|^2\leq \frac{\mathcal E_{2,2(\alpha-3)}(t)}{\alpha-3}\le\frac{\mathcal E_{2,2(\alpha-3)}(t_0)}{\alpha-3},
\end{equation}
and establishes de boundedness of $x$. Integrating \eqref{DE2} from $t_0$ to $t$, and recalling that $\mathcal E_{\lambda,\xi^*}$ is nonnegative, we obtain 
$$\int_{t_0}^t s\|\dot x(s)\|^2 ds \leq \frac{\mathcal E_{2,2(\alpha-3)}(t_0)}{\alpha-3},$$
as required.
\end{proof}

\begin{remark}
In view of \eqref{bornedx} and \eqref{E:x_bounded}, when $\alpha >3$,  we obtain the following explicit bound for $\|\dot x\|$, namely 
$$\|\dot x(t)\| \leq \frac{1}{t}\left(\sqrt{2\mathcal E_{\alpha-1,0}(t_0)} + (\alpha-1)\sqrt{\frac{\mathcal E_{2,2(\alpha-3)}(t_0)}{\alpha-3}}\right).$$
Since $\lim_{t\to+\infty}\|\dot x(t)\|=0$ by Theorem \ref{Thm-weak-conv2}, we also have $\lim_{t\to+\infty}t\,\|\dot x(t)\|^2=0$.
\end{remark}

We are now in a position to prove the weak convergence of the trajectories of \eqref{edo01}, which is the main result of this section:

\begin{Theorem}\label{T:weak_convergence}
Let $\Phi : \mathcal H \rightarrow \mathbb R$ be a continuously differentiable convex function. Let $\argmin\Phi\neq\emptyset$ and let 
$x:[t_0,+\infty[\to\mathcal H$ be a solution of \eqref{edo01} with $\alpha>3$. 
Then $x(t)$ converges weakly, as $t\to+\infty$, to a point in $\argmin\Phi$.
\end{Theorem}

\begin{proof}
We shall use Opial's Lemma \ref{Opial}. To this end, let $x^*\in\argmin\Phi$ and recall from \eqref{E:ineq_h_z} that
$$
\ddot{h}_{x^*}(t) + \frac{\alpha}{t} \dot{h}_{x^*}(t) +\Phi (x(t)) - \min\Phi  \leq \|\dot{x}(t)\|^2,
$$
where $h_z$ is given by \eqref{E:h_z}. This yields
$$t\ddot{h}_{x^*}(t) + \alpha\dot{h}_{x^*}(t)\leq t\|\dot{x}(t)\|^2.$$
In view of  Theorem \ref{T:alpha3}, part ii),
the right-hand side is integrable on $[t_0,+\infty[$. Lemma \ref{basic-edo} then implies $\lim_{t\to+\infty}h_{x^*}(t)$ exists. This gives the first hypothesis in Opial's Lemma. The second one was established in part ii) of Theorem \ref{Thm-weak-conv2}.
\end{proof}

\begin{remark} 
A puzzling question concerns the convergence of the trajectories for $\alpha=3$, a question which is directly related to the convergence of the sequences generated by Nesterov's method. 
\end{remark}

\subsection{Further stabilization results}

Let us complement the study of equation \eqref{edo01} by examining the asymptotic behavior of the acceleration $\ddot{x}$. To this end, we shall use an additional regularity assumption on the gradient of $\Phi$.

\begin{proposition}\label{P:acceleration}
Let $\alpha>3$ and let $x:[t_0,+\infty[ \to\mathcal H$ be a solution of \eqref{edo01} with $\argmin\Phi\neq\emptyset$. Assume $\nabla\Phi$ Lipschitz-continuous on bounded sets. Then $\ddot x$ is bounded, globally Lipschitz continuous on $[t_0,+\infty[$, and satisfies 
$$\lim_{t\to+\infty}\frac{1}{t^\alpha} \int_{t_0}^t s^\alpha \|\ddot x(s)\|^2 ds =0.$$
\end{proposition}

\begin{proof}
First recall that $x$ and $\dot x$ are bounded, by virtue of Theorems \ref{T:alpha3} and \ref{Thm-weak-conv2}, respectively. 
By \eqref{edo01}, we have
\begin{equation}\label{scale-energy7}
\ddot{x}(t) = -\frac{\alpha}{t} \dot{x}(t) - \nabla \Phi (x(t)).
\end{equation}
Since $\nabla \Phi$ is Lipschitz-continuous on bounded sets, it follows from \eqref{scale-energy7}, and the boundedness of $x$ and $\dot x$, that $\ddot {x}$ is bounded on $[t_0, +\infty[$. As a consequence, $\dot {x}$ is Lipschitz-continuous on $[t_0, +\infty[$. Returning to \eqref{scale-energy7}, we deduce that $\ddot{x}$ is Lipschitz-continuous on $[t_0, +\infty[$.

Pick $x^*\in\argmin\Phi$, set $h=h_{x^*}$ (to simplify the notation) and use \eqref{wconv3} to obtain  
\begin{equation} \label{E:h_x*}
\ddot{h}(t) + \frac{\alpha}{t} \dot{h}(t) +  \langle x(t) - x^{*} , \nabla \Phi (x(t)) \rangle =\| \dot{x}(t) \|^2.
\end{equation} 
Let $L$ be a Lipschitz constant for $\nabla\Phi$ on some ball containing the minimizer $x^*$ and the trajectory $x$. 
By virtue of the Baillon-Haddad Theorem (see, for instance, \cite{BaHa},  
\cite[Theorem 3.13]{Pey} or \cite[Theorem 2.1.5]{Nest2}), $\nabla \Phi$ is $\frac{1}{L}$-cocoercive on that ball, which means that 
$$\langle x(t) - x^{*} , \nabla \Phi (x(t))- \nabla \Phi (x^{*})\rangle \geq \frac{1}{L} \| \nabla \Phi (x(t))- \nabla \Phi (x^{*}) \|^2.$$
Substituting this inequality in \eqref{E:h_x*}, and using the fact that $\nabla \Phi (x^{*}) =0$, we obtain
$$ \ddot{h}(t) + \frac{\alpha}{t} \dot{h}(t) + \frac{1}{L} \| \nabla \Phi (x(t)) \|^2  \leq \| \dot{x}(t) \|^2.$$
In view of \eqref{scale-energy7}, this gives
$$\ddot{h}(t) + \frac{\alpha}{t} \dot{h}(t) + \frac{1}{L} \| \ddot{x}(t) +\frac{\alpha}{t} \dot{x}(t) \|^2  \leq \| \dot{x}(t) \|^2.$$
Developing the square on the left-hand side, and neglecting the nonnegative term $(\alpha\|\dot x(t)\|/t)^2/L$, we obtain
$$
\ddot{h}(t) + \frac{\alpha}{t} \dot{h}(t) + \frac{1}{L}\|\ddot{x}(t)\|^2 + \frac{\alpha}{Lt}\frac{d}{dt}\|\dot{x}(t)\|^2 \leq \|\dot{x}(t)\|^2.
$$
We multiply this inequality by $t^\alpha$ to obtain
$$
\frac{d}{dt}\left( t^{\alpha}\dot{h}(t)\right) + \frac{1}{L} t^{\alpha}\|\ddot{x}(t)\|^2 + \frac{\alpha}{L}t^{\alpha-1}\frac{d}{dt}\|\dot{x}(t)\|^2  \leq t^{\alpha} \|\dot{x}(t)\|^2.
$$
Integration from $t_0$ to $t$ yields
$$
t^{\alpha}  \dot{h}(t) - t_0^{\alpha}  \dot{h}(t_0)  +  \frac{1}{L} \int_{t_0}^t  s^{\alpha}\| \ddot{x}(s) \|^2  ds + \frac{\alpha}{L}\Big ( t^{\alpha -1} \|\dot{x}(t) \|^2  
-  {t_0}^{\alpha -1} \|\dot{x}(t_0) \|^2  - (\alpha -1)\int_{t_0}^t \|\dot{x}(s) \|^2 s^{\alpha -2} ds
  \Big ) \leq \int_{t_0}^t   s^{\alpha} \| \dot{x}(s) \|^2 ds.
$$
Neglecting the nonnegative term $\alpha t^{\alpha-1}\|\dot x(t)\|^2/L$, we 
obtain
\begin{equation} \label{E:stabilization}
t^\alpha\dot h(t) + \frac{1}{L} \int_{t_0}^t s^{\alpha}\|\ddot{x}(s)\|^2 ds 
  \leq C + (\alpha -1)\int_{t_0}^t \|\dot{x}(s) \|^2 s^{\alpha -2} ds
  + \int_{t_0}^t   s^{\alpha} \| \dot{x}(s) \|^2 ds,
\end{equation} 
where $C=t_0^\alpha\dot h(t_0)+\alpha t_0^{\alpha-1}\|\dot x(t_0)\|^2/L$.

If $t_0<1$, we have 
$$\frac{1}{t^\alpha} \int_{t_0}^t s^\alpha \|\ddot x(s)\|^2 ds 
= \frac{1}{t^\alpha} \int_{t_0}^1 s^\alpha \|\ddot x(s)\|^2 ds
  + \frac{1}{t^\alpha} \int_1^t s^\alpha \|\ddot x(s)\|^2 ds$$
for all $t\geq 1$. Since the first term on the right-hand side tends to $0$ as $t\to+\infty$, we may assume, without loss of generality, that $t_0\ge 1$.

Observe now that $s^{\alpha -2} \leq s^{\alpha}$, whenever $s\geq1$. Whence, inequality \eqref{E:stabilization} simplifies to
$$t^\alpha\dot h(t) + \frac{1}{L} \int_{t_0}^t s^{\alpha}\|\ddot{x}(s)\|^2 ds 
  \leq C + \alpha \int_{t_0}^t   s^{\alpha} \| \dot{x}(s) \|^2 ds.$$
Dividing by $t^\alpha$ and integrating again, we obtain
$$
 h(t)-h(t_0)  +  \frac{1}{L}   \int_{t_0}^t  \tau^{-\alpha } \left( \int_{t_0}^\tau  {s}^{\alpha}\| \ddot{x}(s) \|^2  ds \right) d\tau
  \leq  \frac{C}{\alpha-1}(t_0^{-\alpha+1}-t^{-\alpha+1})  
    +\alpha \int_{t_0}^t  \tau^{-\alpha } \left( \int_{t_0}^\tau  {s}^{\alpha}\| \dot{x}(s) \|^2  ds \right) d\tau.
$$
Setting $C'=h(t_0)+C t_0^{-\alpha+1}/(\alpha-1)$, and neglecting the 
nonnegative term $h(t)$ of the left-hand side and the nonpositive term 
$-C t^{-\alpha+1}/(\alpha-1)$ of the right-hand side, we get  
$$
 \frac{1}{L}   \int_{t_0}^t  \tau^{-\alpha } \left( \int_{t_0}^\tau  {s}^{\alpha}\| \ddot{x}(s) \|^2  ds \right) d\tau
  \leq  C'
    +\alpha \int_{t_0}^t  \tau^{-\alpha } \left( \int_{t_0}^\tau  {s}^{\alpha}\| \dot{x}(s) \|^2  ds \right) d\tau.
$$
Set $\displaystyle g(\tau)= \tau^{-\alpha}\left(\int_{t_0}^\tau s^\alpha\|\ddot x(s)\|^2 ds\right)$ 
and use Fubini's Theorem on the second integral to get
$$\frac{1}{L}\int_{t_0}^t g(\tau)d\tau \leq  C' + \frac{\alpha}{\alpha-1}\int_{t_0}^t s^\alpha\|\dot x(s)\|^2 (s^{-\alpha+1}-t^{-\alpha+1})\ ds  
\leq  C' + \frac{\alpha}{\alpha-1}\int_{t_0}^t s\|\dot x(s)\|^2 ds. $$
By part ii) of Theorem \ref{T:alpha3}, the integral on the right-hand side is finite. We have 
\begin{equation}\label{gl1}
\int_{t_0}^{+\infty} g(\tau)d\tau < +\infty.
\end{equation}
The derivative of $g$ is
$$\dot g(\tau) =   - \alpha\tau^{-\alpha-1}\int_{t_0}^\tau s^\alpha\|\ddot x(s)\|^2 ds   +\|\ddot x(\tau)\|^2.$$
Let $C''$ be an upper bound for $\|\ddot x\|^2$. We have 
\begin{equation}\label{glip}
|\dot g(\tau)| 
\leq C''\left(1+\alpha\tau^{-\alpha-1}\int_{t_0}^\tau s^\alpha ds\right) 
= C''\left(
  1+\frac{\alpha}{\alpha+1}\tau^{-\alpha-1}
    \left(\tau^{\alpha+1}-t^{\alpha+1}\right)
  \right) 
\leq C''\left(1+\frac{\alpha}{\alpha+1}\right).
\end{equation}
From \eqref{gl1} and \eqref{glip} we deduce that $\lim_{\tau\to+\infty}g(\tau)=0$ by virtue of Lemma \ref{lm:aux}.
\end{proof}

\begin{remark}
Since $\int_{t_0}^ts^\alpha ds= \frac{1}{\alpha+1}\left(t^{\alpha+1}-t_0^{\alpha+1}\right)$, Proposition 
\ref{P:acceleration} expresses a fast ergodic convergence of $\|\ddot x(s)\|^2$ 
to 0 with respect to the weight $s^\alpha$ as $t\to+\infty$, namely 
$$\frac{\int_{t_0}^t s^\alpha\|\ddot x(s)\|^2 ds}{\int_{t_0}^t s^\alpha ds} = 
  o\left(\frac{1}{t}\right).$$
\end{remark}

\section{Strong convergence results}\label{S:strong}

A counterexample due to Baillon \cite{Ba} shows that the trajectories of the steepest descent dynamical system may converge weakly but not strongly. Nevertheless, under some additional geometrical or topological assumptions on $\Phi$, the steepest descent trajectories do converge strongly. This has been proved in the case where the function $\Phi$ is either even or strongly convex (see \cite{Bruck}), or when $\interior(\argmin \Phi) \neq \emptyset$ (see \cite[theorem 3.13]{Bre1}). Some of these results have been extended to inertial dynamics, see \cite{Al} for the heavy ball with friction, and \cite{aabr} for an inertial version of Newton's method. This suggests that convexity alone may not be sufficient for the trajectories of \eqref{edo01} to converge strongly, but one can reasonably expect it to be the case under some additional conditions. The purpose of this section is to establish this fact. The different types of hypotheses will be studied in independent subsections s!
 ince different techniques are required.

\subsection{Set of minimizers with nonempty interior}

Let us begin by studying the case where $\interior(\argmin\Phi)\neq\emptyset$.

\begin{Theorem} \label{Thm-strong-int}
Let $\Phi : \mathcal H \rightarrow \mathbb R$ be a continuously differentiable convex function.  
Let $\interior(\argmin\Phi)\neq\emptyset$ and let $x:[t_0,+\infty[\to\mathcal H$ be a solution of \eqref{edo01} with $\alpha>3$. Then $x(t)$ converges strongly, as $t\to+\infty$, to a point in $\argmin\Phi$. Moreover,
$$\int_{t_0}^{\infty}  t\| \nabla \Phi (x(t))  \|dt < +\infty.$$
\end{Theorem}

\begin{proof}
Since $\interior(\argmin\Phi)\neq\emptyset$, there exist $x^*\in\argmin\Phi$ 
and some $\rho>0$ such that $\nabla\Phi(z)=0$ for all $z\in\mathcal H$ such 
that $\|z-x^*\|<\rho$. By the monotonicity of $\nabla\Phi$, for all 
$y\in\mathcal H$, we have 
$$
\langle \nabla \Phi (y), y- z \rangle \geq 0.
$$
Hence,
$$\langle \nabla \Phi (y), y-x^* \rangle \geq \langle \nabla \Phi (y), z-x^* \rangle .$$
Taking the supremum with respect to $z \in \mathcal H$ such that $\|z-x^*\| < \rho$, we infer that 
$$\langle \nabla \Phi (y), y-x^* \rangle \geq \rho \|\nabla \Phi (y)\|$$
for all $ y \in \mathcal H $. In particular,
$$
\langle \nabla \Phi ( x(t)),  x(t)-x^* \rangle \geq \rho \|\nabla \Phi ( x(t))\|.
$$
By using this inequality in \eqref{dElambdaxi} with $\lambda=\alpha-1$ and 
$\xi=0$, we obtain
$$
\frac{d}{dt}\mathcal E_{\alpha-1,0}(t)+(\alpha-1)\rho t\|\nabla\Phi(x(t))\| 
  \le 2\,t\big(\Phi(x(t))-\min\Phi\big),
$$
whence we derive, by integrating from $t_0$ to $t$
$$
\mathcal E_{\alpha-1,0}(t)-\mathcal E_{\alpha-1,0}(t_0)
+(\alpha-1)\rho\int_{t_0}^t s\|\nabla\Phi(x(s))\|ds 
  \le 2\int_{t_0}^t s\big(\Phi(x(s))-\min\Phi\big)\,ds.
$$
Since $\mathcal E_{\alpha-1,0}(t)$ is nonnegative, part ii) of Theorem \ref{T:SBC} gives
\begin{equation}\label{strong-conv-int8}
\int_{t_0}^{\infty}  t\| \nabla \Phi (x(t))  \|dt < +\infty.
\end{equation}
Finally, rewrite \eqref{edo01} as
\begin{equation*}
t\ddot{x}(t) + \alpha \dot{x}(t) = - t\nabla \Phi (x(t)).
\end{equation*}
Since the right-hand side is integrable, we conclude by applying Lemma \ref{lemma-edo1} and Theorem \ref{T:weak_convergence}.
\end{proof}

\subsection{Even functions}

Let us recall that  $\Phi : \mathcal H \to \mathbb R$ is {\it even} if $\Phi (-x)= \Phi (x) $ for every $x\in \mathcal H$. In this case the set $\argmin \Phi$ is nonempty and contains the origin.  
 
\begin{Theorem} \label{Thm-strong-conv}
Let $\Phi:\mathcal H\rightarrow\mathbb R$ be a continuously differentiable 
convex even function and let $x:[t_0,+\infty[\to\mathcal H$ be a solution of 
\eqref{edo01} with $\alpha>3$. Then $x(t)$ converges strongly, as 
$t\to+\infty$, to a point in $\argmin\Phi$. 
\end{Theorem}

\begin{proof}
For $t_0 \leq \tau \leq s$, set
$$q(\tau)= \| x(\tau) \|^2 -\| x(s) \|^2 - \frac{1}{2} \|x(\tau)- x(s) \|^2 .$$
We have
$$\dot{q}(\tau)= \langle \dot{x}(\tau),x(\tau) + x(s) \rangle \qquad\hbox{and}\qquad \ddot{q}(\tau) = \|\dot{x}(\tau)\|^2 +  \langle \ddot{x}(\tau),x(\tau) + x(s) \rangle.$$
Combining these two equalities and using \eqref{edo01}, we obtain
\begin{equation} \label{E:q_ddot}
\ddot{q}(\tau) + \frac{\alpha}{\tau} \dot{q}(\tau) =   \|\dot{x}(\tau)\|^2 +  \langle \ddot{x}(\tau) + \frac{\alpha}{\tau} \dot{x}(\tau)    ,x(\tau) + x(s) \rangle =    \|\dot{x}(\tau)\|^2 -  \langle    \nabla \Phi (x(\tau))    ,x(\tau) + x(s) \rangle .
\end{equation}
Recall that the energy 
$W(\tau)= \frac{1}{2} \| \dot{x}(\tau) \|^2  + \Phi (x(\tau))$
is nonincreasing. Therefore,
\begin{eqnarray*}
\frac{1}{2} \| \dot{x}(\tau) \|^2  + \Phi (x(\tau)) & \geq  &  \frac{1}{2} \| \dot{x}(s) \|^2  + \Phi (x(s)) \\
& = &  \frac{1}{2} \| \dot{x}(s) \|^2  + \Phi (-x(s))\\
& \ge & \frac{1}{2} \| \dot{x}(s) \|^2  + \Phi (x(\tau)) - \langle  \nabla \Phi (x(\tau)) ,x(\tau) + x(s) \rangle,
\end{eqnarray*}
by convexity. After simplification, we obtain
\begin{equation}\label{even6}
\frac{1}{2} \| \dot{x}(\tau) \|^2   \geq  -  \langle    \nabla \Phi (x(\tau)),x(\tau) + x(s) \rangle  .
\end{equation}
Combining \eqref{E:q_ddot} and \eqref{even6}, we obtain
$$\tau\ddot{q}(\tau) + \alpha\dot{q}(\tau) \leq   \frac{3}{2}\tau \|\dot{x}(\tau)\|^2.$$
As in the proof of Lemma \ref{basic-edo}, we have
$$
\dot{q}(\tau) \leq k(\tau):= \frac{C}{\tau^{\alpha} }  +  \frac{3}{2\tau^{\alpha} }  \int_{t_0}^\tau  u^{\alpha} \|\dot{x}(u)\|^2du,
$$
where $C=2\|\dot x(t_0)\|\,\|x\|_{\infty}$. The function $k$ does not depend on $s$. Moreover, using Fubini's Theorem, we deduce that
$$
\int_{t_0}^{+\infty}k(\tau)\,d\tau\le \frac{C}{t_0^{\alpha-1}(\alpha-1)}+\frac{3}{2(\alpha-1)}\int_{t_0}^{+\infty}u\|\dot x(u)\|^2\,du<+\infty,
$$
by part ii) of Theorem \ref{T:alpha3}. Integrating $\dot{q}(\tau) \leq k(\tau)$ from $t$ to $s$, we obtain
$$
 \frac{1}{2} \|x(t)- x(s) \|^2 \leq  \| x(t) \|^2 -\| x(s) \|^2 
+ \int_t^s k(\tau) d\tau.
$$
Since $\Phi$ is even, we have $0 \in\argmin \Phi$. Hence 
$\lim_{t\to +\infty }\| x(t) \|^2$ exists (see the proof of Theorem \ref{T:weak_convergence}). As a consequence, $x(t)$ has the Cauchy property as $t \to + \infty$, and hence converges.
\end{proof}

\subsection{Uniformly convex functions}

Following \cite{BC}, a function $\Phi : \mathcal H \rightarrow \mathbb R$ is {\it uniformly convex on bounded sets} if, for each $r>0$, there is an increasing function $\omega_r:[0,+\infty[\to[0,+\infty[$ vanishing only at $0$, and such that
\begin{equation} \label{E:uniform}
\Phi(y) \geq \Phi (x) + \langle  \nabla \Phi (x)  , y-x \rangle + \omega_r(\|x-y\|)
\end{equation}
for all $x, y \in \mathcal H$ such that $\| x \|\leq r$ and $\| y \|\leq r$. Uniformly convex functions are strictly convex and coercive.

\begin{Theorem} \label{T:unif_convex}
Let $\Phi$ be uniformly convex on bounded sets, and let $x:[t_0,+\infty[\to\mathcal H$ be a solution of \eqref{edo01} with $\alpha>3$. Then $x(t)$ converges strongly, as $t\to+\infty$, to the unique $x^*\in\argmin\Phi$.
\end{Theorem}

\begin{proof}
Recall that the trajectory $x(\cdot)$ is bounded, by part ii) in Theorem \ref{T:alpha3}. Let $r >0$ be such that $x$ is contained in the ball of radius $r$ centered at the origin. This ball also contains $x^*$, which is the weak limit of the trajectory in view of the weak lower-semicontinuity of the norm and Theorem \ref{T:weak_convergence}. Writing $y=x(t)$ and $x=x^*$ in \eqref{E:uniform}, we obtain
$$\omega_r (\|x(t)-x^{*} \|)\le \Phi(x(t))-\min\Phi. $$
The right-hand side tends to $0$ as $t\to+\infty$ by virtue of Theorem \ref{Thm-weak-conv2}. It follows that $x(t)$ converges strongly to $x^*$ as $t\to+\infty$.
\end{proof}

Let us recall that a function $\Phi : \mathcal H \rightarrow \mathbb R$ is {\it strongly convex} if there exists a positive constant $\mu$  such that
$$  
\Phi(y) \geq \Phi (x) + \langle  \nabla \Phi (x)  , y-x \rangle + \frac{\mu}{2} \|x-y  \|^2
$$
for all $x, y \in \mathcal H$. Clearly, strongly convex functions are uniformly convex on bounded sets. However, a striking fact is that convergence rates increase indefinitely with larger values of $\alpha$ for these functions.

\begin{Theorem}
Let $\Phi:\mathcal H\rightarrow\R$ be strongly convex, and let $x:[t_0,+\infty[\to\mathcal H$ be a solution of \eqref{edo01} with $\alpha>3$. Then $x(t)$ converges strongly, as $t\to+\infty$, to the unique element $x^*\in\argmin\Phi$. Moreover 
\begin{equation}\label{elliptique}
\Phi(x(t))-\min\Phi=\mathcal O\left(t^{-\dta}\right),\qquad 
\|x(t)-x^*\|^2=\mathcal O\left(t^{-\dta}\right),\qquad\hbox{and}\qquad
\|\dot x(t)\|^2=\mathcal O\left(t^{-\dta}\right).
\end{equation}
\end{Theorem}

\begin{proof}
Strong convergence follows from Theorem \ref{T:unif_convex} because strongly convex functions are uniformly convex on bounded sets.
From \eqref{dElambdap} and the strong convexity of $\Phi$, we deduce that
\begin{eqnarray*}
\frac{d}{dt}\Elambdap(t) & \leq & 
  (p+2-\lambda)t^{p+1}(\Phi(x)-\min\Phi) 
  - \lambda(\alpha-\lambda-1-p)t^p\langle x-x^*,\dot x\rangle \\
\nonumber
 & & 
  -\frac{\lambda}{2}(\mu t^2-p\lambda)t^{p-1}\|x-x^*\|^2
  - \left(\alpha-\lambda-1-\frac{p}{2}\right)t^{p+1}\|\dot x\|^2
\end{eqnarray*}
for any $\lambda\geq 0$ and any $p\geq 0$. Now fix $p=\dtat$ and $\lambda=\dta$, so that $p+2-\lambda=\alpha-\lambda-1-p/2=0$ and $\alpha-\lambda-1-p=-p/2$. The above inequality becomes 
$$\frac{d}{dt}\Elambdap(t) \leq \frac{\lambda p}{2}t^p\langle x-x^*,\dot x\rangle -\frac{\lambda}{2}(\mu t^2-p\lambda)t^{p-1}\|x-x^*\|^2.$$
Define $t_1=\max\left\{t_0,\sqrt{\frac{p\lambda}{\mu}}\right\}$, so that
$$\frac{d}{dt}\Elambdap(t) \leq \frac{\lambda p}{2}t^p\langle x-x^*,\dot x\rangle$$
for all $t\geq t_1$. Integrate this inequality from $t_1$ to $t$ (use integration by parts on the right-hand side) to get 
$$
\Elambdap(t) \leq \Elambdap(t_1) + \frac{\lambda p}{4}\left( t^p\|x(t)-x^*\|^2 - t_1^p\|x(t_1)-x^*\|^2 - p\int_{t_1}^t s^{p-1}\|x(s)-x^*\|^2 ds \right).
$$
Hence,
\begin{equation}\label{Ett}
\Elambdap(t) \leq \Elambdap(t_1) + \frac{\lambda p}{4} t^p\|x(t)-x^*\|^2 \leq \Elambdap(t_1) + \frac{\lambda p}{2\mu} t^p(\Phi(x(t))-\min\Phi),
\end{equation}
in view of the strong convexity of $\Phi$. By the definition of $\Elambdap$, we have
$$
t^{p+2}((\Phi(x(t))-\min\Phi) \leq \Elambdap(t) \leq 
\Elambdap(t_1) + \frac{\lambda p}{2\mu} t^p(\Phi(x(t))-\min\Phi).
$$
Dividing by $t^{p+2}$ and using the definition of $t_1$, along with the fact that $t\geq t_1$, we obtain
\begin{eqnarray*}
\Phi(x(t))-\min\Phi & \leq &   \Elambdap(t_1)t^{-p-2} + \frac{\lambda p}{2\mu} t^{-2}(\Phi(x(t))-\min\Phi)\\
 & \leq & \Elambdap(t_1)t^{-p-2}+\frac{\lambda p}{2\mu}t_1^{-2}(\Phi(x(t))-\min\Phi)\\
 & \leq & \Elambdap(t_1)t^{-p-2}+\frac{1}{2}(\Phi(x(t))-\min\Phi).
\end{eqnarray*}
Recalling that $p=\dtat$ and $\lambda=\dta$, we deduce that
\begin{equation}\label{Phit}
\Phi(x(t))-\min\Phi \leq 2\Elambdap(t_1)t^{-p-2} = \left[2\mathcal E_{\frac{2}{3}\alpha}^{\frac{2}{3}(\alpha-3)}(t_1)\right]t^{-\frac{2}{3}\alpha}.
\end{equation}
The strong convexity of $\Phi$ then gives
\begin{equation}\label{xxt}
\|x(t)-x^*\|^2 \leq \frac{2}{\mu}(\Phi(x(t))-\min\Phi) \leq \left[\frac{4}{\mu}\Elambdap(t_1)\right]t^{-p-2} = 
  \left[\frac{4}{\mu} \mathcal E_{\frac{2}{3}\alpha}^{\frac{2}{3}(\alpha-3)}(t_1)\right]t^{-\frac{2}{3}\alpha}.
\end{equation}
Inequalities \eqref{Phit} and \eqref{xxt} settle the first two points in \eqref{elliptique}.

Now, using \eqref{Ett} and \eqref{Phit}, we derive
$$\Elambdap(t) \leq \Elambdap(t_1) + \frac{\lambda p}{2\mu} t^p(\Phi(x(t))-\min\Phi)\le \Elambdap(t_1)+\frac{\lambda p}{\mu}\Elambdap(t_1)t^{-2}\le \Elambdap(t_1)+\frac{\lambda p}{\mu}\Elambdap(t_1)t_1^{-2}\le 2\Elambdap(t_1).$$
The definition of $\Elambdap$ then gives
$$\frac{t^p}{2}\|\lambda(x(t)-x^*)+t\dot x(t)\|^2 \le \Elambdap(t) \le 2\Elambdap(\tau).$$
Hence
$$\|\lambda(x(t)-x^*)+t\dot x(t)\|^2 \leq 4 t^{-p} \Elambdap(t_1),$$
and
$$t\|\dot x(t)\| \leq  2 t^{-p/2} \sqrt{\Elambdap(t_1)} + \lambda\|x(t)-x^*\|.$$
But using \eqref{xxt}, we deduce that
$$\lambda\|x(t)-x^*\|\le \frac{2\lambda}{\sqrt{\mu}}t^{-p/2-1}\sqrt{\Elambdap(t_1)}.$$
The last two inequalities together give
$$t\|\dot x(t)\| \le   2 t^{-p/2} \sqrt{\Elambdap(t_1)} \left(1+\frac{\lambda t^{-1}}{\sqrt{\mu}}\right) \le 2 t^{-p/2}\sqrt{\Elambdap(t_1)}\left(1+\sqrt{\frac{\lambda}{p}}\right).$$
Taking squares, and rearranging the terms, we obtain
$$\|\dot x(t)\|^2 \leq \left[4 \left(1+\sqrt{\frac{\alpha}{\alpha-3}}\right)^2\mathcal E_{\frac{2}{3}\alpha}^{\frac{2}{3}(\alpha-3)}(t_1)
    \right] t^{-\frac{2}{3}\alpha},$$
which shows the last point in \eqref{elliptique} and completes the proof.
\end{proof}

The preceding theorem extends \cite[Theorem 4.2]{SBC}, which states that if $\alpha>9/2$, then $\Phi(x(t))-\min\Phi=\mathcal O(1/t^{3})$.

%
%
%

\section{Convergence of the associated algorithms} \label{S:algorithm}

In many situations, one is faced with a non-smooth convex minimization problems with an additive structure of the form
\begin{equation}\label{algo1}
\min \left\lbrace  \Phi (x) + \Psi (x): \ x \in \mathcal H \right\rbrace,
\end{equation}
where $\Phi: \mathcal H \to  \mathbb R \cup \lbrace   + \infty  \rbrace  $ is proper, lower-semicontinuous and convex, and $\Psi: \mathcal H \to  \mathbb R$ is convex and continuously differentiable.

Following the analysis carried out in the previous sections, it seems reasonable to consider the differential inclusion
\begin{equation}\label{algo2}
\ddot{x}(t) + \frac{\alpha}{t} \dot{x}(t)  + \partial \Phi (x(t)) + \nabla \Psi (x(t)) \ni 0,
\end{equation}
in order to approximate optimal solutions for \eqref{algo1}. This differential inclusion is a special instance of
$$\ddot{x}(t) + a(t) \dot{x}(t)  + \partial \Theta (x(t)) \ni 0,$$
where $\Theta: \mathcal H \to  \mathbb R \cup \lbrace   + \infty  \rbrace  $  is a convex lower-semicontinuous proper function, and $a(\cdot)$ is a positive damping parameter. This differential inclusion has been studied in \cite{ACR} in the case of a fixed positive damping parameter $a(t)\equiv\gamma >0$. In that setting, and at least localy, each trajectory is Lipschitz continuous, its velocity has bounded variation, and its acceleration is a bounded vectorial measure.

Thus, setting $\Theta (x)= \Phi (x) + \Psi (x)$,  we can reasonably expect that the rapid convergence properties, studied in Section \ref{S:minimizing} for the solutions of \eqref{edo01}, should hold for the solutions of \eqref{algo2} as well. Especially, that $\Theta(x(t))-\min\Theta=\mathcal O(1/t^2)$, and that each trajectory converges to an optimal solution. A detailed analysis is an interesting topic for further research, but goes beyond the scope of this paper.
 
Using these ideas as a guideline, we shall introduce corresponding fast converging algorithms, making the link with Nesterov \cite{Nest1}-\cite{Nest4}, Beck-Teboulle \cite{BT}, and the recent works of Chambolle-Dossal \cite{CD}, and Su-Boyd-Candes \cite{SBC}. More precisely, it is possible to discretize \eqref{algo2} implicitely with respect to the nonsmooth function $\Phi$, and explicitly with respect to the smooth function $\Psi$. Indeed, taking a time step size $h>0$, and $t_k= kh$, $x_k = x(t_k)$ the classical finite difference scheme for (\ref{algo2}) gives
\begin{equation}\label{algo3}
\frac{1}{h^2}(x_{k+1} -2 x_{k}   + x_{k-1} ) +\frac{\alpha}{kh^2}( x_{k}  - x_{k-1})  + \partial \Phi (x_{k+1}  ) + \nabla \Psi (y_k) \ni 0,
\end{equation}
where  $y_k$ is a linear combination of $x_k$ and $x_{k-1}$, that will be made precise later. After developing \eqref{algo3}, we obtain 
\begin{equation}\label{algo4}
x_{k+1} + h^2 \partial \Phi (x_{k+1}) \ni  x_{k} + \left( 1- \frac{\alpha}{k}\right) ( x_{k}  - x_{k-1}) - h^2 \nabla \Psi (y_k).
\end{equation}
A natural choice for $y_k$ leading to a simple formulation of the algorithm is
\begin{equation}\label{algo5}
y_k=   x_{k} + \left( 1- \frac{\alpha}{k}\right) ( x_{k}  - x_{k-1}).
\end{equation}
Of course, other choices are possible, an interesting topic for further research. Using the classical {\it proximity operator}
$$\mbox{prox}_{ \gamma \Phi } (x)= {\argmin}_{\xi} \left\lbrace \Phi (\xi) + \frac{1}{2 \gamma} \| \xi -x  \|^2
\right\rbrace  = \left(I + \gamma \partial \Phi \right)^{-1} (x),$$
and setting $\gamma=h^2$, the algorithm can be written as
\begin{equation}\label{algo7}
 \left\{
\begin{array}{l}
y_k=   x_{k} + \left( 1- \frac{\alpha}{k}\right) ( x_{k}  - x_{k-1});  	   \\
\rule{0pt}{20pt}
 x_{k+1} = \mbox{prox}_{\gamma\Phi} \left( y_k- \gamma \nabla \Psi (y_k) \right).
 \end{array}\right.
\end{equation}
For practical purposes, and in order to fit with the existing literature on the subject, it is convenient to work with the following equivalent formulation
\begin{equation}\label{algo7b}
\left\{
\begin{array}{l}
y_k=   x_{k} + \frac{k -1}{k  + \alpha -1} ( x_{k}  - x_{k-1});  	   \\
\rule{0pt}{20pt}
 x_{k+1} = \mbox{prox}_{\gamma\Phi} \left( y_k- \gamma\nabla \Psi (y_k) \right). 
 \end{array}\right.
\end{equation}

This algorithm is within the scope of the proximal-based inertial algorithms \cite{AA1}, \cite{MO}, \cite{LP} and forward-backward methods. It has been recently introduced by Chambolle-Dossal \cite{CD}, and Su-Boyd-Cand\`es \cite{SBC}. For $\alpha = 3$, we recover the classical algorithm based on Nesterov and G\"uler ideas, and developed by Beck-Teboulle (FISTA)
\begin{equation}\label{algo7c}
 \ \left\{
\begin{array}{l}
y_k=   x_{k} + \frac{k -1}{k  + 2} ( x_{k}  - x_{k-1});  	   \\
\rule{0pt}{20pt}
 x_{k+1} = \mbox{prox}_{h^2 \Phi} \left( y_k- h^2 \nabla \Psi (y_k) \right). 
 \end{array}\right.
\end{equation}
The fast convergence properties of the algorithm \eqref{algo7b} were recently highlighted by Su-Boyd-Cand\`es \cite{SBC} and Chambolle-Dossal \cite{CD}. An important $-$ and still open $-$ question regarding the FISTA method, as described in \eqref{algo7c}, is the convergence of sequences $(x_k)$ and $(y_k)$. The main interest of considering the broader context given in \eqref{algo7b} is that, for $\alpha >3$, these sequences  converge. This has recently been obtained by Chambolle-Dossal \cite{CD}.
Following \cite{SBC}, we will see that the proof of the convergence properties of \eqref{algo7b} can be obtained in a parallel way with the convergence analysis in the continuous case.

\bigskip

More precisely, following the arguments in the preceding sections, one is able to prove the following:

\begin{Theorem}
If $\alpha>0$, then $\lim\Phi(x_k)=\inf\Phi$ and every weak limit point of $(x_k)$, as $k\to+\infty$, belongs to $\argmin\Phi$. For $\argmin\Phi\neq\emptyset$, we have the following:
\begin{itemize}
	\item [i)] If $\alpha\ge 3$, then $\Phi(x_k)-\min\Phi=\mathcal O(1/k^2)$ and $\|x_{k+1}-x_k\|=\mathcal O(1/k)$.
	\item [ii)] If $\alpha>3$, then $\sum k(\Phi(x_k)-\min\Phi)<+\infty$, $\sum k\|x_{k+1}-x_k\|^2<+\infty$, and $x_k$ converges weakly, as $k\to+\infty$, to some $x^*\in\argmin\Phi$. Strong convergence holds if $\Phi$ is even, uniformly convex, or if $\argmin\Phi\neq\emptyset$.
\end{itemize}
\end{Theorem}

\if{

\begin{proof} 

The convergence analysis in the continuous case provides us a guide for the study of the discrete case.
Following this dynamical approach, let us sketch the main lines of the proof, see \cite{SBC} and  \cite{CD} for a detailed proof.

\smallskip

To simplify notations let us set $\Theta = \Phi + \Psi$, ${\Theta}^* = \inf (\Phi + \Psi)$, $s=h^2$. 

\noindent \emph{Step one:} Let us introduce the following energy function, the correspondent of the function $ \mathcal E_{\alpha-1,0} $ defined in section 
\ref{DefEnerg}, which will serve  as  base for the Lyapunov analysis:
\begin{equation}\label{algo12}
   \mathcal E (k)= 2s \frac{(k+\alpha -2)^2}{\alpha -1} (\Theta (x_k) - {\Theta}^*) + \frac{1}{\alpha -1} \| (k+\alpha -1)y_k -ky_k - (\alpha -1) x^*  \|^2 .
\end{equation}
Then, it is proved in \cite[Theorem 4.3]{SBC} that $ \mathcal E (k)$ is a strict Lyapunov function, and, for any  $k \in \mathbb N$
\begin{equation}\label{algo13}
   \mathcal E (k)+  \frac{2s}{\alpha -1} \Big ( (\alpha -3) (k+ \alpha -2)  +1 \Big ) (\Theta (x_k) -  {\Theta}^* ) \leq \mathcal E (k-1) .
\end{equation}
From this, and the nonnegativity of the different constituents of $  \mathcal E (k)$, we immediately infer the fast convergence property 
(\ref{algo8}), and the estimation (\ref{algo9})
\begin{equation*}
  \sum_k    k \left((\Phi + \Psi)(x_k) - \inf(\Phi + \Psi) \right)   < + \infty.
\end{equation*}
This is the equivalent of the estimate (\ref{energy1}) in the continuous case (with $\Theta$ instead of $\Phi$)
\begin{equation*}
  \int_{t_0}^{\infty}  t \left( \Theta(x(t)) - \inf\Theta \right) dt  < + \infty.
\end{equation*}

\smallskip

\noindent \emph{Step two:} The next step consists in passing from (\ref{algo9}) to the energy estimate (\ref{algo10})
\begin{equation*}
  \sum_k    k  \| x_k - x_{k-1} \|^2  dt   < + \infty ,
\end{equation*}
which is the discrete version of the continuous energy estimate (\ref{energy1})
\begin{equation*}
  \int_{t_0}^{\infty}  t\| \dot{x}(t) \|^2  dt   < + \infty .
\end{equation*}
Estimation (\ref{algo10}) has been obtained in \cite[Corollary 2]{CD}.

\smallskip

\noindent \emph{Step three:} The final step is to apply Opial's Lemma.  Using the previous estimates, it is a direct adaptation of the classical proof of the convergence of proximal-like  inertial algorithms. It is a parallel argument to that using the differential inequality
(\ref{basic-edo1}) with $\|x_k- x^{*} \|^2$ instead of $\|x(t)- x^{*} \|^2$, and $x^{*} \in \argmin (\Phi + \Psi)$.

\end{proof}

}\fi

\section{Further remarks} \label{S:conclusion}

\subsection{Nonsmooth objective function}

As mentioned at the beginning of Section \ref{S:algorithm}, it is interesting to establish the asymptotic properties, as $t\to+\infty$, of the solutions of the differential inclusion \eqref{algo2}. Beyond global existence issues, one must check that the Lyapunov analysis is still valid. In view of the validity of the subdifferential inequality for convex functions, the (generalized) chain rule for derivatives over curves (see \cite{Bre1}), our conjecture is that most results presented here can be transposed to this more general context, except for the stabilization of the acceleration, which relies on the Lipschitz character of the gradient.

\subsection{Time reparameterization} Let us examine briefly the effect of some simple rescaling procedures:

\subsubsection*{Invariance}
The condition $\alpha >3$ is not affected by an affine time rescaling. Indeed, if $a>0$ and we take $t=as$ in \eqref{edo01}, we obtain
$$\ddot{y}(s) + \frac{\alpha}{s} \dot{y}(s) + a^2\nabla \Phi (y(s)) = 0,$$
where $y(s)= x(as)$. This produces an analogue system with $\Phi$ replaced by $\Phi_a:=a^2\Phi$.

\subsubsection*{Variable damping versus variable mass}
If we take $t=\sqrt{s}$ in \eqref{edo01} we obtain
$$ 2s\ddot{z} (s) + (2\alpha + 1) \dot{z}(s) + \nabla \Phi (z(s)) = 0,$$
where $z(s)= x(\sqrt{s})$. In this alternative formulation, the viscous damping parameter is fixed, but the mass coefficient becomes infinitely large as $t\to + \infty$. This suggests that a parallel analysis can be performed by controlling the mass coefficient, instead of the viscosity coefficient.

\subsection{Selection of the initial conditions}
The constant in the order of convergence given by Theorem \ref{T:SBC} is 
$$K(x_0,v_0)=\mathcal E_{\alpha-1,0}(t_0)=t_0^2(\Phi(x_0)-\min\Phi)+\frac{1}{2}\|(\alpha-1)(x_0-x^*)+t_0v_0\|^2,$$
where $x_0=x(t_0)$ and $v_0=\dot x(t_0)$. This quantity is minimized when $x_0^*\in\argmin\Phi$ and $v_0^*=\frac{(\alpha-1)}{t_0}(x^*-x_0^*)$, with $\min K=0$. If $x_0^*\neq x^*$, the trajectory will not be stationary, but the value $\Phi(x(t))$ will be constantly equal to $0$. Of course, selecting $x_0^*\in\argmin\Phi$ is not realistic, and the point $x^*$ is unknown. Keeping $\hat x_0$ fixed, the function $v_0\mapsto K(\hat x_0,v_0)$ is minimized at $\hat v_0=\frac{(\alpha-1)}{t_0}(x^*-\hat x_0)$. This suggests taking the initial velocity as a multiple of an approximation of $x^*-\hat x_0$, such as the gradient direction $\hat v_0=\nabla\Phi(\hat x_0)$, Newton or Levenberg-Marquardt direction $\hat v_0=[\varepsilon I+\nabla^2\Phi(\hat x_0)^{-1}]\nabla\Phi(\hat x_0)$ ($\varepsilon\ge 0$), or the proximal point direction $\hat v_0=\left[(I+\gamma\nabla\Phi)^{-1}(\hat x_0)-\hat x_0\right]$ ($\gamma>>0$).

\subsection{Continuous versus discrete}

The analysis carried out in Section \ref{S:algorithm} for inertial forward-backward algorithm \eqref{algo7b} is a reinterpretation of the proof of the corresponding results in the continuous case. In other words, we built a complete proof having the continuous setting as a guideline. It would be interesting to know whether the results in \cite{Alv_Pey2,Alv_Pey3} can be applied in order to deduce the asymptotic properties without repeating the proofs.

\subsection{Hessian-driven damping}

In the dynamical system studied here, second-order information with respect to time ultimately induces fast convergence properties. On the other hand, in Newton-type methods, second-order information in space, has a similar consequence. In a forthcoming paper, we analyze the solutions of the second-order evolution equation
\begin{equation*}
 \ddot{x}(t) + \frac{\alpha}{t} \dot{x}(t) + \beta\,\nabla^2 \Phi (x(t))\,\dot{x} (t) + \nabla \Phi (x(t)) = 0,
\end{equation*}
where $\Phi$ is a smooth convex function, and $\alpha$, $\beta$ are positive parameters. This inertial system combines an isotropic viscous damping which vanishes asymptotically, and
a geometrical Hessian-driven damping, which makes it naturally related to Newton and Levenberg-Marquardt methods.

\section{Appendix: Some auxiliary results}

In this section, we present some auxiliary lemmas to be used later on. The following result can be found in \cite{AAS}:

\begin{lemma} \label{lm:aux} 
Let $\delta >0$, $1\leq p<\infty$ and $1\leq r\leq\infty$. Suppose $F\in L^{p}([\delta,\infty[)$
is a locally absolutely continuous nonnegative function, $G\in L^{r}([\delta,\infty[)$
and 
$$
\frac{d}{dt}F(t)\leq G(t)
$$
for almost every $t>\delta$. Then $\lim_{t\to\infty}F(t)=0$. 
 \end{lemma}

To establish the weak convergence of the solutions of \eqref{edo01}, we will use Opial's Lemma \cite{Op}, that we recall in its continuous form. This argument was first used in \cite{Bruck} to establish the convergence of nonlinear contraction semigroups.

\begin{lemma}\label{Opial} Let $S$ be a nonempty subset of $\mathcal H$ and let $x:[0,+\infty)\to \mathcal H$. Assume that 
\begin{itemize}
\item [(i)] for every $z\in S$, $\lim_{t\to\infty}\|x(t)-z\|$ exists;
\item [(ii)] every weak sequential limit point of $x(t)$, as $t\to\infty$, belongs to $S$.
\end{itemize}
Then $x(t)$ converges weakly as $t\to\infty$ to a point in $S$.
\end{lemma}

The following allows us to establish the existence of a limit for a real-valued function, as $t\to+\infty$:

\begin{lemma}\label{basic-edo} 
Let $\delta >0$, and let $w: [\delta, +\infty[ \rightarrow \mathbb R$ be a continuously differentiable function which is bounded from below. Assume
\begin{equation}\label{basic-edo1}
t\ddot{w}(t) + \alpha \dot w(t) \leq g(t),
\end{equation}
for some $\alpha > 1$, almost every $t>\delta$, and some nonnegative function $g\in L^1 (\delta, +\infty)$. Then, the positive part $[\dot w]_+$ of $\dot w$ belongs to $L^1(t_0,+\infty)$ and $\lim_{t\to+\infty}w(t)$ exists.
\end{lemma}
 
\begin{proof} 
Multiply \eqref{basic-edo1} by $t^{\alpha-1}$ to obtain
$$
\frac{d}{dt} \big(t^{\alpha} \dot w(t)\big) \leq t^{\alpha-1} g(t).
$$
By integration, we obtain
$$
\dot w(t) \leq \frac{{\delta}^{\alpha}|\dot w(\delta)|}{t^{\alpha} }  +  \frac{1}{t^{\alpha} }  \int_{\delta}^t  s^{\alpha-1} g(s)ds.
$$
Hence, 
$$
[\dot w]_{+}(t)  \leq \frac{{\delta}^{\alpha}|\dot w(\delta)|}{t^{\alpha} }  +  \frac{1}{t^{\alpha} }  \int_{\delta}^t  s^{\alpha-1} g(s)ds,
$$
and so,
$$
\int_{\delta}^{\infty} [\dot w]_{+}(t)  dt \leq \frac{{\delta}^{\alpha}|\dot w(\delta)|}{(\alpha -1) \delta^{\alpha -1}} +
  \int_{\delta}^{\infty}\frac{1}{t^{\alpha}}  \left(  \int_{\delta}^t  s^{\alpha-1} g(s) ds\right)  dt.
$$
Applying Fubini's Theorem, we deduce that
$$
\int_{\delta}^{\infty}\frac{1}{t^{\alpha}}  \left(  \int_{\delta}^t  s^{\alpha-1} g(s) ds\right)  dt =     
\int_{\delta}^{\infty}  \left(  \int_{s}^{\infty}  \frac{1}{t^{\alpha}} dt\right) s^{\alpha-1} g(s) ds 
= \frac{1}{\alpha -1} \int_{\delta}^{\infty}g(s) ds.
$$
As a consequence,
$$
 \int_{\delta}^{\infty} [\dot w]_{+}(t) dt \leq \frac{{\delta}^{\alpha}|\dot w(\delta)  |}{(\alpha -1) \delta^{\alpha -1}}  +
 \frac{1}{\alpha -1} \int_{\delta}^{\infty}g(s)  ds < + \infty.
$$
Finally, the function $\theta:[\delta,+\infty)\to\R$, defined by
$$\theta(t)=w(t)-\int_{\delta}^{t}[\dot w]_+(\tau)\,d\tau,$$
is nonincreasing and bounded from below. It follows that
$$\lim_{t\to+\infty}w(t)=\lim_{t\to+\infty}\theta(t)+\int_{\delta}^{+\infty}[\dot w]_+(\tau)\,d\tau$$ 
exists.
\end{proof}

The following is a vector-valued version of Lemma \ref{basic-edo}:

\begin{lemma}\label{lemma-edo1} 
Take $\delta >0$, and let $F \in L^1 (\delta , +\infty; \mathcal H)$ be continuous. Let $x:[\delta, +\infty[ \rightarrow \mathcal H$ be a solution of
\begin{equation}\label{strong-conv-int10}
t\ddot{x}(t) + \alpha \dot{x}(t)  = F(t)
\end{equation}
with $\alpha >1$. Then, $x(t)$ converges strongly in $\mathcal H$ as $t \to + \infty$.
 \end{lemma}
\begin{proof}
As in the proof of Lemma \ref{basic-edo}, multiply \eqref{strong-conv-int10} by $t^{\alpha -1}$ and integrate to obtain
$$\dot{x}(t) = \frac{{\delta}^{\alpha}\dot{x}(\delta)}{t^{\alpha}} + 
\frac{1}{t^{\alpha}} \int_{\delta}^t s^{\alpha -1}F(s)ds.$$
Integrate again to deduce that
$$x(t) = x(\delta) + {\delta}^{\alpha}\dot{x}(\delta)  \int_{\delta}^t   \frac{1}{s^{\alpha}} ds + \int_{\delta}^t
\frac{1}{s^{\alpha}} \left( \int_{\delta}^s {\tau}^{\alpha -1}F(\tau)d\tau\right) ds.$$
Fubini's Theorem applied to the last integral gives
\begin{equation}\label{strong-conv-int14}
x(t) = x(\delta) +   \frac{{\delta}^{\alpha}\dot{x}(\delta)}{\alpha-1}\left(\frac{1}{\delta^{\alpha-1}}-\frac{1}{t^{\alpha-1}}\right)+ 
\frac{1}{\alpha -1} \left( \int_{\delta}^t F(\tau)d\tau
- \frac{1}{t^{\alpha-1}}   \int_{\delta}^t {\tau}^{\alpha-1} F(\tau)d\tau \right).
\end{equation}
Finally, apply Lemma \ref{basic-int} to the last integral with $\psi(s)=s^{\alpha-1}$ and $f(s)=\|F(s)\|$ to conclude that all the terms in the right-hand side of \eqref{strong-conv-int14} have a limit as $t \to + \infty$.
\end{proof}

The following is a continuous version of Kronecker's Theorem for series (see, for example, \cite[page 129]{Kno}):

\begin{lemma}\label{basic-int} 
Take $\delta >0$, and let $f \in L^1 (\delta , +\infty)$ be nonnegative and continuous. Consider a nondecreasing function $\psi:(\delta,+\infty)\to(0,+\infty)$ such that $\lim\limits_{t\to+\infty}\psi(t)=+\infty$. Then, 
$$\lim_{t \rightarrow + \infty} \frac{1}{\psi(t)} \int_{\delta}^t \psi(s)f(s)ds =0.$$ 
\end{lemma}
 \begin{proof} 
Given $\epsilon >0$, fix $t_\epsilon$ sufficiently large so that 
$$\int_{t_\epsilon}^{\infty}  f(s) ds \leq \epsilon.$$
Then, for $t \geq t_\epsilon$, split the integral $\int_{\delta}^t \psi(s)f(s) ds$ into two parts to obtain
$$\frac{1}{\psi(t)} \int_{\delta}^t \psi(s)f(s)ds =
\frac{1}{\psi(t)}\int_{\delta}^{t_\epsilon} \psi(s) f(s) ds
+ \frac{1}{\psi(t)}\int_{t_\epsilon}^t \psi(s) f(s) ds 
\leq   \frac{1}{\psi(t)}\int_{\delta}^{t_\epsilon} \psi(s)f(s) ds
+ \int_{t_\epsilon}^t  f(s) ds.$$
Now let $t\to+\infty$ to deduce that
$$0\le\limsup_{t\to+\infty}\frac{1}{\psi(t)}\int_{\delta}^t \psi(s)f(s)ds \le \epsilon.$$
Since this is true for any $\epsilon>0$, the result follows.
\end{proof}



\begin{thebibliography}{10}

\bibitem{AAS} {\sc B. Abbas, H. Attouch, B. F. Svaiter}, {\it   Newton-like dynamics
and forward-backward methods for structured monotone inclusions in Hilbert spaces}, J. Optim. Theory Appl., 161 (2014),  No. 2, pp. 331-360.


\medskip


 \bibitem{AAC}{\sc S. Adly, H. Attouch, A. Cabot}, {\it  Finite time stabilization of nonlinear oscillators subject to dry friction},
Nonsmooth Mechanics and Analysis (edited by P. Alart, O. Maisonneuve and R.T. Rockafellar),
Adv. in Math. and Mech., Kluwer (2006), pp.~289--304.

\medskip

\bibitem{Al}{\sc F. Alvarez},  {\it  On the minimizing property of a second-order dissipative system in Hilbert spaces},
SIAM J. Control Optim., 38, No. 4, (2000), pp. 1102-1119. 

\medskip

\bibitem{AA1} {\sc F. Alvarez, H. Attouch}, {\it  An inertial proximal method for maximal monotone operators via discretization of a nonlinear oscillator with damping},
     Set-Valued Analysis,  9 (2001), No. 1-2, pp.  3--11.
     
\medskip

 \bibitem{AA} {\sc F. Alvarez, H. Attouch}, {\it  Convergence and asymptotic stabilization for some damped hyperbolic equations with non-isolated equilibria},
ESAIM Control Optim. Calc. of Var.,  6 (2001), pp.  539--552.

\medskip

\bibitem{aabr}{\sc F. Alvarez, H. Attouch, J. Bolte, P. Redont}, {\it A second-order gradient-like dissipative dynamical system with Hessian-driven damping. Application to optimization and mechanics},   J. Math. Pures Appl.,  81 (2002), No. 8, pp.  747--779.
    
\medskip

\bibitem{Alv_Pey2} {\sc F. Alvarez, J. Peypouquet}, {\it Asymptotic almost-equivalence of Lipschitz evolution systems in Banach spaces},  Nonlinear Anal.,  73 (2010), No. 9, pp. 3018--3033.

\medskip

\bibitem{Alv_Pey3} {\sc F. Alvarez, J. Peypouquet}, {\it A unified approach to the asymptotic almost-equivalence of evolution systems without Lipschitz conditions}, Nonlinear Anal. 74 (2011), No. 11, pp. 3440--3444.

\medskip

\bibitem{ABM}{\sc H. Attouch, G. Buttazzo, G. Michaille}, {\it  Variational analysis in Sobolev and BV spaces. Applications to PDE's and optimization},  MPS/SIAM Series on Optimization, 6, Society for Industrial and Applied Mathematics (SIAM), Philadelphia, PA, Second edition, 2014, 793 pages.
 
 \medskip
  
 \bibitem{ACR} {\sc H. Attouch, A.  Cabot, P. Redont}, {\it  The dynamics of elastic shocks via epigraphical
 regularization of a differential inclusion},   Adv. Math. Sci. Appl.,  12   (2002), No.1,  pp.  273--306.
 
  \medskip

\bibitem{AtCz1}{\sc H. Attouch, M.-O. Czarnecki}, {\it  Asymptotic control and stabilization 
of nonlinear oscillators with non-isolated equilibria}, J. Differential Equations, 179 (2002), pp.~278--310.
 
\medskip


\bibitem{AGR}{\sc H. Attouch, X. Goudou, P. Redont},  {\it  The heavy ball with friction method. The continuous dynamical system, global exploration of the local minima
 of a real-valued function by asymptotical analysis of a dissipative dynamical system},
 Commun. Contemp. Math., 2  (2000), No. 1, pp.~1--34.  

\medskip

\bibitem{AMR} {\sc H. Attouch, P.E. Maing\'e, P. Redont},  {\it  A second-order differential system with Hessian-driven damping; Application to non-elastic shock laws},
  Differential Equations and Applications,  4 (2012), No. 1, pp. 27--65.

\medskip

\bibitem{APR} {\sc H. Attouch, J. Peypouquet, P. Redont},  {\it  A dynamical approach to an inertial forward-backward algorithm for convex minimization},
SIAM J. Optim., 24  (2014), No. 1, pp.~232--256. 



\medskip

\bibitem{AS}{\sc H. Attouch, A. Soubeyran}, {\it  Inertia and reactivity in 
decision making as cognitive variational inequalities}, Journal of Convex 
Analysis, 13 (2006), pp.~207-224.

\medskip

\bibitem{Ba}{\sc J.-B. Baillon}, {\it  Un exemple concernant le comportement asymptotique de la solution du probl\`eme $\frac{du}{dt} + \partial \phi(u) \ni 0$}, Journal of Functional Analysis, 28 (1978), pp.~369-376.

\medskip
\bibitem{BaHa}{\sc J.-B. Baillon, G. Haddad}, {\it  Quelques
propri\'et\'es des op\'erateurs angle-born\'es et n-cycliquement monotones},
Israel J. Math., 26 (1977), pp. 137-150 .

\medskip

\bibitem{BC}{\sc H. Bauschke, P. Combettes}, {\it  Convex Analysis and Monotone Operator Theory in Hilbert spaces}, CMS Books in Mathematics, Springer,   (2011).


\medskip

\bibitem{BT}{\sc A. Beck, M. Teboulle},  {\it  A fast iterative shrinkage-thresholding algorithm for linear inverse problems},  SIAM J. Imaging Sci., 2  (2009),  No. 1, pp.~183--202.

\medskip

\bibitem{Bre1}{\sc H. Br\'ezis}, {\it  Op\'erateurs maximaux monotones dans les 
espaces de Hilbert et \'equations d'\'evolution}, Lecture Notes 5, North Holland, (1972).

\medskip

\bibitem{Bre2}{\sc H. Br\'ezis}, {\it  Asymptotic behavior of some evolution systems: Nonlinear evolution equations},
 Academic Press, New York, (1978), pp.~141--154.
 
\medskip 

\bibitem{Bruck} {\sc R.E. Bruck},  {\it   Asymptotic convergence of nonlinear contraction
semigroups in Hilbert spaces}, J. Funct. Anal.,  18  (1975),
pp.~15--26. 

 
 \medskip
 \bibitem{Cabot-inertiel}{\sc A. Cabot}, {\it  Inertial gradient-like dynamical system controlled by a stabilizing term}, J. Optim. Theory Appl., 120 (2004), pp.~275--303.
 

\medskip


 \bibitem{CEG1}{\sc  A. Cabot, H. Engler, S. Gadat}, {\it  On the long time behavior of second order differential equations
with asymptotically small dissipation}
Transactions of the American Mathematical Society, 361 (2009), pp.~5983--6017.

\medskip

\bibitem{CEG2}{\sc  A. Cabot, H. Engler, S. Gadat}, {\it  Second order differential equations with asymptotically small dissipation
and piecewise flat potentials},
Electronic Journal of Differential Equations, 17 (2009), pp.~33--38. 

\medskip


\bibitem{CD}{\sc  A. Chambolle, Ch. Dossal}, {\it  On the convergence of the iterates of Fista},
HAL Id: hal-01060130
https://hal.inria.fr/hal-01060130v3
Submitted on 20 Oct 2014.

\medskip


\bibitem{Kno} {\sc K. Knopp K} ``Theory and application of infinite series". Blackie \& Son, Glasgow, 1951.


\medskip


\bibitem{MO}{\sc A. Moudafi, M. Oliny}, {\it  Convergence of a splitting inertial proximal method for monotone operators}, J. Comput. Appl. Math., 155  (2003), No. 2, pp.   447--454.

\medskip

\bibitem{Nest1}{\sc  Y. Nesterov}, {\it   A method of solving a convex programming problem with convergence rate
O(1/k2)}, Soviet Mathematics Doklady,  27  (1983), pp.~ 372--376.

\medskip

\bibitem{Nest2}{\sc  Y. Nesterov}, {\it  Introductory lectures on convex optimization: A basic course}, volume 87 of
Applied Optimization. Kluwer Academic Publishers, Boston, MA, 2004.

\medskip

\bibitem{Nest3}{\sc  Y. Nesterov}, {\it  Smooth minimization of non-smooth functions}, Mathematical programming,
103 (2005), No. 1, pp.~127--152.

\medskip

\bibitem{Nest4}{\sc  Y. Nesterov}, {\it  Gradient methods for minimizing composite objective function}, CORE Discussion
Papers, 2007.

\medskip


\bibitem{Op} {\sc Z. Opial}, {\it  Weak convergence of the sequence of successive approximations for nonexpansive mappings},  Bull. Amer. Math.
Soc.,  73  (1967), pp. 591--597.

\medskip

\bibitem{DC}{\sc  B. O'Donoghue, E. J. Cand\`es}, {\it  Adaptive restart for accelerated gradient schemes}, Foundations of Computational Mathematics, 15  (2015), No. 3,  pp. 715--732.

\medskip

\bibitem{Pey}{\sc J. Peypouquet}, {\it Convex optimization in normed spaces: theory, methods and examples}. Springer, 2015.

\medskip

\bibitem{Pey_Sor} {\sc J. Peypouquet, S. Sorin}, {\it Evolution equations for maximal monotone operators: asymptotic analysis in continuous and discrete time}, J. Convex Anal, 17  (2010), No. 3-4, pp. 1113--1163.

\medskip

\bibitem{LP}{\sc D. A. Lorenz, T. Pock}, {\it  An inertial forward-backward algorithm
for monotone inclusions}, J. Math. Imaging Vision, pp. 1-15, 2014. (online).

\medskip

\bibitem{SBC}{\sc W.  Su,  S. Boyd,  E. J. Cand\`es}, {\it A Differential Equation for Modeling Nesterov's Accelerated Gradient Method: Theory and Insights}. Neural Information Processing Systems (NIPS) 2014. 






\end{thebibliography}
\end{document}